\documentclass{amsart}
\usepackage{graphicx}
\usepackage{mathrsfs}
\usepackage{amssymb}

\usepackage{color}

\newtheorem{thm}{Theorem}[section]
\newtheorem{cor}[thm]{Corollary}

\newtheorem{prop}[thm]{Proposition}
\theoremstyle{definition}
\newtheorem{defn}[thm]{Definition}
\theoremstyle{remark}
\newtheorem{rem}[thm]{Remark}
\numberwithin{equation}{section}

\newcommand{\Hom}{\mathrm{Hom}}

\newcommand{\h}{\mathrm{H}}
\newcommand{\pr}{\mathbf{proj}\mathcal{A}_\theta}
\newcommand{\vc}{\mathbf{Vect}(\mathbb{T}_\theta)}

\begin{document}

\title{Motives of noncommutative tori}
\author{Yunyi Shen}
\address{Mathematics Department, Florida State University, 1017 Academic Way, Tallahas-
see, FL 32306-4510, USA}
\email{yshen2@math.fsu.edu}


\begin{abstract}
In this article, 
we propose a way of seeing
the noncommutative tori in the category of noncommutative motives. As an algebra, the noncommutative torus is lack the smoothness property required to define a noncomutative motive. Thus, instead of working with the algebra, we work with the category of holomorphic bundles.
It is known that these are related to the coherent sheaves of an elliptic curve. We describe the cyclic homology of the category of holomorphic bundle on a noncommutative torus. We then introduce a notion of (weak) t-structure in dg categories. By applying the t-structure to a noncommutative torus, we show that it induces a decomposition of the motivic Galois group of the Tannakian subcategory generated by the auxiliary elliptic curve.
\end{abstract}
\maketitle
\section{Introduction}
In algebraic geometry, the notion of motives was first suggested by Grothendieck, as a possible approach to the Weil conjectures. There is at present, due to the result of Jannsen \cite{J}, a good construction of an abelian category of pure movies for smooth projective varieties. For more general varieties, there are several candidates of triangulated categories having the required properties of the bounded derived category of mixed motives, due to work of Voevodsky, Hanamura, and Levine \cite{V}, \cite{H}, \cite{L}, though in general the existence of mixed motives is still conjectural \cite{MVW}.

At the beginning of the the 21st century, Kontsevich introduced motives to the noncommutative world. He first proposed a  construction of  a category of noncommutative numerical motives in the frame of dg categories \cite{Ko1} and a possible notion of noncomutative Chow motives. Then Tabuada constructed the category of noncommutative Chow motives in \cite{Tab1} and showed how it is related to the classical category of Chow motives in the commutative sense. Recently, a series of papers on the noncommutative motives written by Marcolli and Tabuada \cite{MT1}, \cite{MT2}, \cite{MT3} further developed this theory. Their papers extend some of the classical results about pure motives to the noncommutative framework, such as the semi-simplicity of the category of numerical motives, standard conjectures on pure motives, Artin motives, etc. Several properties of classical motives are reformulated in terms of noncommutative motives.

The noncommutative motives successfully extend pure motives to noncommutative (algebraic) geometry. A natural question is : can we observe the typical noncommutative objects (like noncommutative torus) in this construction? One major problem might be how we can relate these objects to algebraic spaces \cite{Ko1}, because as algebras they are usually not smooth,
or even Noetherian.

However, for noncommutative tori, the situation is much better. Thanks to the remarkable work of Schwartz and Polishchuk \cite{PS}, \cite{P}, we can use holomorphic bundles on the noncommutative torus. They have shown the category of holomorphic bundles is equivalent to the heart of a nonstandard t-structure on some elliptic curve. We can associate a noncommutative motive to a smooth projective variety. So our question becomes how can we describe the nonstandard t-structure in the category of noncommutative motives. In this article, we will introduce a notion of a (weak) t-structure over dg categories. Moreover, we will see the additive invariant will split
into direct sum on the t-structure.

This article will be organized as following: in section 1, we will review some basic concepts in dg categories and introduce the noncommutative motives. In section 2, we will first review the category of holomorphic bundles on a noncommutative torus and then discuss its cyclic homology. In the last section, we will introduce the (weak) t-structure in dg categories and discuss
the t-structure in noncommutative motives. Then we apply the properties given by the t-structure to the noncommutative tori.

Through this article, we fix $k$ as a field.
\section{Preliminaries}
\subsection{Dg categories} For dg categories, our main reference is Keller's article \cite{Ke1}.

A dg (stands for differential graded) category $\mathcal{C}$ is a $k$-category whose morphisms $\mathcal{C}(X,Y)$ are a complex of modules over $k$ and a composition of two morphisms is a morphism of complexes over $k$
$$\mathcal{C}(X,Y)\otimes\mathcal{C}(Y,Z)\rightarrow\mathcal{C}(X,Z).$$
Where $X,\,Y$ and $Z$ are objects of $\mathcal{C}$.

For example, let $A$ be a $k$-algebra, then we can define the dg category $\mathcal{C}_{dg}(A)$ as the following: the objects are complexes of right $A$ modules and for any two complexes $M=(M^m,d_M),\,N=(N^m,d_N)$ we define the $n$-th component $\mathcal{C}_{dg}(A)(M,N)^n$ of the hom sets as morphisms of $A$-modules $f:M^m\rightarrow N^{m+n}$, $\forall m\in\mathbb{Z}$ whose differential is given by $d(f)=d_M\circ f-(-1)^nf\circ d_N$. The composition of two morphisms in $\mathcal{C}_{dg}(A)$ is just the composition of graded maps.

Given a dg category $\mathcal{C}$, we can immediately construct two new categories associated to $\mathcal{C}$, the categories $Z^0(\mathcal{C})$ and $H^0(\mathcal{C})$. They are the categories with the same objects as in $\mathcal{C}$ but with morphisms $Z^0(\mathcal{C})(X,Y)=Z^0(\mathcal{C}(X,Y))$ and $H^0(\mathcal{C})(X,Y)=H^0(\mathcal{C}(X,Y))$ for objects $X$ and $Y$.

A dg functor is a functor between two dg categories and which is a morphism of complexes of modules over $k$ between the morphism sets. All the small dg categories then form the category $\mathrm{dgcat}_k$. From now on, we always assume a dg category is an object in $\mathrm{dgcat}_k$.

For $\mathcal{C},\,\mathcal{C}'\in\mathrm{dgcat}_k$, we say the dg functor $Q:\,\mathcal{C}\rightarrow\mathcal{C}'$ is a quasi-equivalence if $Q(X,Y): \mathcal{C}(X,Y)\rightarrow\mathcal{C}'(Q(X),Q(Y))$ is a quasi-isomorphism on complexes over $k$ and the functor $H^0(Q):\,H^0(\mathcal{C})\rightarrow H^0(\mathcal{C}')$ is an equivalence of categories.

A left (right) dg $\mathcal{C}$-module is a dg functor $\mathcal{C}\rightarrow\mathcal{C}_{dg}(k)$ ($\mathcal{C}^{op}\rightarrow\mathcal{C}_{dg}(k)$, where $\mathcal{C}^{op}$ is the opposite category of $\mathcal{C}$. Naturally, for every object $X$ in $\mathcal{C}$, we see $X^\wedge=\mathcal{C}(?,\,X)$ as a right dg module. A morphism between (left or right) dg modules is a natural transformation that is compatible with the dg structures for every object of $\mathcal{C}$. So the dg $\mathcal{C}$-modules form a category. A morphism of two dg modules is a quasi-isomorphism if it induces an isomorphism in homologies for every object. If we localize the category of dg $\mathcal{C}$-modules with respect to the quasi-isomorphisms, we get the derived category of $\mathcal{C}$ which is denoted by $\mathcal{D}(\mathcal{C})$.

It is known that $\mathcal{D}(\mathcal{C})$ is a triangulated category. Recall that an object $X$ of $\mathcal{D}(\mathcal{C})$  is compact if and only if $\Hom_{\mathcal{D}(\mathcal{C})}(X,\,?)$ commutes with arbitrary coproducts in $\mathcal{D}(\mathcal{C})$. All the compact objects form the full subcategory $\mathcal{D}_c(\mathcal{C})$.

\subsection{The dg enhancement}
For any dg category $\mathcal{C}$, there is a dg category embedding $\mathcal{C}\hookrightarrow\textrm{Pre-Tr}(\mathcal{C})$. $\textrm{Pre-Tr}(\mathcal{C})$ is called the \emph{pretriangulated hull} of $\mathcal{C}$, which is the smallest dg category containing $\mathcal{C}$ and admitting the shift functor $X[1]$ and the mapping cone $C_f$ for any morphism in $Z^0(\textrm{Pre-Tr}(\mathcal{C}))$. If we set $\mathrm{Tr}(\mathcal{C})=H^0(\textrm{Pre-Tr}(\mathcal{C}))$, then $\mathrm{Tr}(\mathcal{C})$ is a triangulated category with the shift functor and mapping cones inherited from $\textrm{Pre-Tr}(\mathcal{C})$. The concrete construction can be found in \cite{BK}.

A dg category is called \emph{pretriangulated} if the imbedding $\mathcal{C}\hookrightarrow\textrm{Pre-Tr}(\mathcal{C})$ is also a quasi-equivalence. So the equivalence of categories $H^0(\mathcal{C})\rightarrow \mathrm{Tr}(\mathcal{C})$ makes $H^0(\mathcal{C})$ a triangulated category.

Let $\mathcal{K}$ be a triangulated category. We say $\mathcal{K}$ is dg enhanced if there is a pretriangulated dg category $\mathcal{C}$ and an equivalence of triangulated categories $H^0(\mathcal{C})\rightarrow \mathcal{K}$. In this case, $\mathcal{C}$ is called a dg enhancement of $\mathcal{K}$.

For example, let $\mathcal{A}$ be an abelian category with enough injectives (like quasi-coherent sheaves on schemes). It is not hard to show the derived category $\mathscr{D}^b(\mathcal{A})$ is equivalent to the full subcategory in the homotopy category $\mathscr{K}^+(\mathcal{A})$ consisting of the complexes whose terms are all injectives and whose homology groups are almost all zeros. As showed in \cite{BK}, we know there is a full subcategory in $\mathcal{C}^+_{dg}(\mathcal{\mathcal{A}})$ (cf. the example in section 2.1) enhancing $\mathscr{D}^b(\mathcal{A})$ and we will denote this enhancement by $\mathcal{D}^b_{dg}(\mathcal{A})$.

\subsection{The noncommutative motives and Tannakian categories}
Two dg categories $\mathcal{B}$, $\mathcal{C}$ are said to be Morita equivalent if there is a dg functor $\mathcal{B}\rightarrow\mathcal{C}$ that gives an equivalence on the derived categories $\mathcal{D}(\mathcal{B})\rightarrow\mathcal{D}(\mathcal{C})$. Even more, we can make $\mathrm{dgcat}_k$ a Quillen model category with the Morita equivalences as its weak equivalences and its homotopy category is denoted by $\mathrm{Hmo}_k$. Moreover, a tensor product $\otimes^{\mathbf{L}}$ can be defined on $\mathrm{Hmo}_k$, with which we can make $\mathrm{Hmo}_k$ a symmetric monoidal category. \cite{CT}

We also define the category $\mathrm{Hmo}_0$ as the category which has the same objects as $\mathrm{Hmo}_k$ and whose morphisms $\mathcal{B}\rightarrow\mathcal{C}$ are given by the Grothendieck group of $\mathrm{rep}(\mathcal{B},\mathcal{C})$. $\mathrm{rep}(\mathcal{B},\mathcal{C})$ is subcategory of bimodules $X$ in $\mathcal{D}(\mathcal{B}^{op}\otimes^{\mathbf{L}}\mathcal{C})$ such that $X(?,B)$ is in $\mathcal{D}_c(\mathcal{C})$, for any object $B$ of $\mathcal{B}$. There is a canonical functor $\mathrm{Hmo}_k\rightarrow\mathrm{Hmo}_0$. Let $\mathcal{A}$ be an additive category, then the functor $F:\mathrm{Hmo}_k\rightarrow\mathcal{A}$ is said an additive invariant if $F$ factors through $\mathrm{Hmo}_k\rightarrow\mathrm{Hmo}_0$ \cite{Ke1}. 

According to Kontsevich \cite{Ko1}, we have the following important definitions: A dg category $\mathcal{C}$ is called smooth if the bimodule $\mathcal{C}(?,\,?)$ is a compact $\mathcal{C}^{op}\otimes^\mathbf{L}\mathcal{C}$ module. It is called proper if for any objects $X,\,Y$ of $\mathcal{C}$, the complex $\mathcal{C}(X,Y)$ is perfect.

The category of noncommutative Chow motives, $\mathrm{NChow}(k)_F$, with coefficients in the field $F$ is defined as the pseudo-abelian hull of the category whose objects are smooth and proper dg categories, whose morphisms from dg category $\mathcal{B}$ to $\mathcal{C}$ are the $F$-coefficient Grothendieck group $\mathbf{K}_0(\mathcal{B}^{op}\otimes^\mathbf{L}\mathcal{C})_F$, and where the composition of two morphisms is induced by $\otimes^\mathbf{L}$ on bimodules \cite{Tab1}.

Similarly to what happens in the theory of motives in classical algebraic geometry, we still call a morphism in the category of noncommutative Chow motives a correspondence. Moreover, one can define the intersection number on these correspondences, and similarly to the case of algebraic cycles, one can define a numerical equivalence relation on them. If we mod out these relations, we get a new category. It is called the category of noncommutative numerical motives and is denoted by $\mathrm{NNum}(k)_F$. The main result in the paper \cite{MT1} shows that $\mathrm{NNum}(k)_F$ is abelian and semisimple.

Let $X$ be a projective variety over $k$, $D^b(\mathrm{Coh}(X))$ be the bounded derived category of coherent sheaves on $X$, and $per(X)$ the subcategory of perfect complexes on $X$. When $X$ is smooth in addition, it is well-known that $D^b(\mathrm{Coh}(X))$ is equivalent to $per(X)$. In this case, we also know $D^b(\mathrm{Coh}(X))$ (=$per(X)$) has a dg enhancement, denoted by $D^b_{per}(X)$. According to the result of To\"{e}n \cite{TV}, the dg category $D^b_{per}(X)$ is smooth and proper. So for any smooth projective variety $X$, $D^b_{per}(X)$ gives an object of the noncommutative Chow motives $\mathrm{NChow}(k)_F$. In this way, the category of classical pure Chow motives are "embedded" (mod the Tate twists) the into $\mathrm{NChow}(k)_F$ and so are the numerical motives \cite{Tab2}. 

Let $F$ be a field, and $L/F$ a field extension. Let $\mathcal{A}$ be
a rigid symmetric monoidal $F$-linear abelian category. An $L$-valued \emph{fibre functor} is a faithful exact tensor functor $\omega: \mathcal{A}\rightarrow Vect_L$. Where $Vect_L$ means the category of finite generated $L$-modules. We then have the definitions: i), $\mathcal{A}$ is called a \emph{Tannakian} category if there is an $L$-valued fibre functor on $\mathcal{A}$; ii), $\mathcal{A}$ is called a \emph{neutral Tannakian} category if there is an $F$-valued fibre functor on $\mathcal{A}$. When $\mathcal{A}$ is neutral Tannakian, we can define the $F$-algebraic group, which is called the Galois group of $\mathcal{A}$, as $Gal(\mathcal{A})(R)=\{\mathrm{isomorphisms\,of\,tensor\,functors}\,\,\omega\otimes R\rightarrow\omega\otimes R\}$, for an $F$-algebras $R$. If we let $G_{\mathcal{A}}=Gal(\mathcal{A})$ and $Rep_F(G_{\mathcal{A}})$ be the category of finite dimensional $F$-representations of $G_\mathcal{A}$, then there is an equivalence of categories $\mathcal{A}\tilde{\rightarrow}Rep_F(G_\mathcal{A})$.

In Grothendieck's theory of pure motives, there are several standard conjectures. They are named as standard conjecture B, C, D and I. When the standard conjectures C and D hold, one can make the category of pure numerical motives (the so-called Grothendieck motives) a neutral Tannakian category by changing the signs of the symmetric isomorphisms and using a Weil cohomology as the fiber functor \cite{Mi}.

Now, with the same notations above, we call a tensor functor $\omega: \mathcal{A}\rightarrow sVect_L$ a $L$-valued \emph{super-fibre functor} if it is exact and faithful. Where $sVect_L$ stands for the finite dimensional super ($\mathbb{Z}_2$-graded) vector spaces over $L$. Similarly, one can give the definitions of \emph{super-Tannakian} and \emph{neutral super-Tannakian} categories (please consult \cite{D}, also Appendix A of \cite{MT2}). 

It has been shown in \cite{MT2}, the category $\mathrm{NNum}(k)_F$ is super-Tannakian. Even more, in the same paper, Marcolli and Tabuada stated the standard conjecture C and D in noncommutative motives. Moreover, they proved if the noncommutative standard conjecture C and D hold, then $\mathrm{NNum}(k)_F$ can be modified to a neutral Tannakian category.

\section{Holomorphic bundles on the noncommutative tori}
\subsection{Noncommutative tori and holomorphic bundles}
If we now fix an irrational real number $\theta$, then the algebra $\mathcal{A}_\theta$ of smooth functions on the noncommutative torus $\mathbb{T}_\theta$ is $$\mathcal{A}_\theta=\bigg\{\sum_{(m,n)\in\mathbb{Z}}a_{m,n}U_1^mU_2^n\mid a_{m,n}\in\mathcal{S}(\mathbb{Z}^2)\bigg\}$$
where $\mathcal{S}(\mathbb{Z}^2)$ denotes the Schwartz space on $\mathbb{Z}^2$ and $U_1$, $U_2$ satisfy the relation $U_1U_2=e^{2\pi i\theta}U_2U_1$.

A vector bundle on $\mathbb{T}_\theta$ is a right finitely generated projective $\mathcal{A}_\theta$-module. We denote by $\pr$ the category of vector bundles on $\mathbb{T}_\theta$.

If $\tau=\tau_1+i\tau_2\in\mathbb{C}$ with $\tau_1,\tau_2\in\mathbb{R}$ and $\tau_2\neq0$, then there is a complex structure on $\mathbb{T}_\theta$ given by the derivation $\delta_\tau$ on $\mathcal{A}_\theta$ such that $$\delta_\tau(U_1)=2\pi\tau i,\quad\delta_\tau(U_2)=2\pi i.$$ From now on, we use the notation $\mathbb{T}_{\theta,\tau}$ to indicate that there is a complex structure $\delta_\tau$ on $\mathbb{T}_\theta$.

Let $E$ be an object in $\pr$. A \emph{holomorphic structure} on $E$ is a map $\nabla:E\rightarrow E$ such that $\nabla(ea)=\nabla(e)a+e\delta_\tau(a)$ for $e\in E, a\in\mathcal{A}_\theta$. We may say that the pair $(E,\nabla)$ is a \emph{holomorphic bundle} on $\mathbb{T}_\theta$. A morphism between two holomorphic bundles is a morphism of $\mathcal{A}_\theta$-modules which is also compatible with the holomorphic structures. We also denote the category of holomorphic bundles by the notation $\vc$. Given a pair of integers $(m,n)$ such that $n+\theta m\neq0$,  in addition to satisfying $m\neq0$, there is a holomorphic bundle $(E_{n,m},\nabla_z)$ in $\vc$, such that $E_{n,m}$ is the Schwartz space $\mathcal{S}(\mathbb{R}\times\mathbb{Z}\big/m\mathbb{Z})$ and $\nabla_z(f)=\frac{\partial f}{\partial x}+2\pi i(\tau\mu x+z)f$, where $\mu=\frac{m}{n+\theta m}$ and $z$ is a complex number. When $m=0$, we let $E_{n,0}=\mathcal{A}_\theta^{|n|}$ and define the holomorphic structure as $\nabla_z(a)=\delta_\tau(a)+2\pi iza$. For any vector bundle $E$ on $\mathbb{T}_\theta$, $E$ is isomorphic to $E_{n,m}^{\oplus k}$ for some $E_{n,m}$ \cite{R}. So, we can naturally give a holomorphic structure on $E$. However, the converse is not always true. Not every holomorphic bundle is given by a direct sum of these $(E_{n,m},\nabla_z)$'s.

The holomorphic bundles are closely related to elliptic curves. In fact,
according to a remarkable result of Schwarz and Polishchuck \cite{PS}, \cite{P}, the category of holomorphic vector bundles on a noncommutative torus is equivalent to the heart of some t-structure on $D^b(X)$ for some $\theta$.

More precisely, let $X=X_\tau=\frac{\mathbb{C}}{\mathbb{Z}+\mathbb{Z}\tau}$ be an elliptic curve and $\mathrm{Coh}(X)$ be the category of coherent sheaves on $X$. For any real number $\theta$, we denote by $\mathrm{Coh}_{>\theta}$ the subcategory of all the coherent sheaves whose semistable factors all have slopes $>\theta$ and $\mathrm{Coh}_{\leq\theta}$ as the subcategory of all the coherent sheaves whose semistable factors all have slopes $\leq\theta$. Then, $(\mathrm{Coh}_{>\theta},\mathrm{Coh}_{\leq\theta})$ is a torsion pair of $\mathrm{Coh}(X)$. If we let $D^b(X)=D^b(\mathrm{Coh}X)$ and continue to define $$D^{\theta,\leq0}=\{F\in D^b(X)\mid H^{>0}(F)=0,\,H^0(F)\in \mathrm{Coh}_{>\theta}\},$$$$D^{\theta,\geq1}=\{F\in D^b(X)\mid H^{<0}(F)=0,\,H^0(F)\in \mathrm{Coh}_{\leq\theta}\},$$ then $(D^{\theta,\leq0},\,D^{\theta,\geq0})$ is a t-structure on $D^b(X)$. We denote the heart of this t-structure by $\mathcal{C}^\theta$.

A heart of a t-structure is always an abelian category. So immediately, the category of holomorphic vector bundles on a noncommutative torus is abelian.

\subsection{Cyclic homology of holomorphic bundles}

It is well-known that for any $k$-algebra $A$, the category of finitely generated projective modules over $A$ is an exact category. By this comment we know $\pr$ is an exact category. According to last paragraph we also know that $\vc$ is an abelian category. The following proposition gives a relation between exact sequences in these two categories.
\begin{prop}Let $F:\mathrm{Vect}(\mathbb{T}_{\theta})\rightarrow \mathrm{proj}\mathcal{A}_\theta$ be the functor that forgets the holomorphic structure on the vector bundle. Thus $F$ is faithful and exact. Moreover every exact sequence in $\mathrm{proj}\mathcal{A}_\theta$ comes from an exact sequence in $\mathrm{Vect}(\mathbb{T}_\theta)$.
\end{prop}
\begin{proof} It is easy to see that $F$ is faithful and exact.

We first show that if $f:E_1\rightarrow E_2$ is surjective in $\pr$, then if there is a holomorphic structure $\nabla_2$ on $E_2$, then we can find a holomorphic structure $\nabla_1$ on $E_1$, which makes $f$ a morphism in $\vc$.

We consider the case where $E_1$ is free at first and then we generalize it to the finitely generated projective case.

If $E_1$ is free and has basis $\{e_1,\ldots,e_n\}$, then $\{f(e_1)\ldots f(e_n)\}$ spans $E_2$. Moreover, if $\nabla_2(f(e_i))=\sum_jf(e_j)b_i^j$, we simply define $\nabla_1(e_i a^i)=e_i\delta_\tau(a^i)+\sum_je_jb_i^ja^i$. One can easily check that this is a holomorphic structure on $E_1$, which
makes $f$ a morphism in $\vc$.

Now we consider the case where $E_1$ is a finitely generated projective $\mathcal{A}_\theta$-module. Therefore, $E_1$ is a direct summand of a finitely generated free module $E_0$, with the projection $p:E_0\rightarrow E_1$. Moreover, the composition map $f\circ p: E_0\rightarrow E_2$ is a surjection. Thus, as it is showed in the previous case, there is a holomorphic structure $\nabla_0$ on $E_0$ making $f\circ p$ a morphism in $\vc$. We define $\nabla_1=p\nabla_0$. We will see that $\nabla_1$ is a holomorphic structure on $E_1$, which makes $f:(E_1,\nabla_1)\rightarrow(E_2,\nabla_2)$ a morphism in $\vc$.

We have $p\nabla_0(ea)=p(\nabla_0(e)a+e\delta_\tau(a))=p\nabla_0(e)a+e\delta_\tau(a)$. This determines a holomorphic structure on $E_1$. We also have
$f(p\nabla_0(e))=f\circ p(\nabla_0(e))=\nabla_2(f\circ p(e))=\nabla_2(f(e))$. Hence $f$ is compatible with these holomorphic structures.

If $0\rightarrow E_0\rightarrow E_1\rightarrow E_2\rightarrow0$ is a short exact sequence in $\pr$ and $(E_2,\nabla_2)$ is an object in $\vc$, then there is a $\nabla_1$ on $E_1$ making $E_1\rightarrow E_2$ a morphism in $\vc$. Moreover if we see $E_0$ as a kernel, then $E_0$ naturally has a holomorphic structure $\nabla_0$ inherited from $(E_1,\nabla_1)$. Thus, $0\rightarrow(E_0,\nabla_0)\rightarrow(E_1,\nabla_1)\rightarrow(E_2,\nabla_2)\rightarrow0$ is an short exact sequence in $\vc$ and with the same morphisms if we forget the holomorphic structures.

For a long exact sequence $0\rightarrow E_0\rightarrow E_1\rightarrow\cdots\rightarrow E_n\rightarrow0$, if we break it into short exact sequences and apply the argument above, then we have a long exact sequence in $\vc$. The morphisms remain the same, when we forget the holomorphic structures.

\end{proof}

Now let us turn to the cyclic homology of these categories. In the spirit of Mitchell, for a dg category $\mathcal{B}$, its mixed complex $C(\mathcal{B})$ is given by the following constructions:
\begin{itemize}
\item[i)] the n-th term of precyclic complex:$$\coprod_{B_0,\ldots,B_n\in\mathrm{Ob}(\mathcal{B})}\mathcal{B}(B_n,B_0)\otimes\mathcal{B}(B_{n-1},B_n)\otimes\cdots\otimes\mathcal{B}(B_0,B_1);$$
\item[ii)] the degeneracy maps and the cyclic operator: $$\partial_i(f_n,\ldots,f_i,f_{i-1},\ldots,f_0)=\begin{cases}(f_n,\ldots,f_if_{i-1},\ldots,f_0)&i\neq0\\(-1)^{\alpha}(f_0f_n,\ldots,f_1)&i=0\end{cases},$$ $$t(f_{n-1},\ldots,f_0)=(-1)^\alpha(f_0,f_{n-1},f_{n-2},\ldots,f_1),$$ where $\alpha=n+\mathrm{deg}f_0\cdot(\mathrm{deg}f_1+\cdots+\mathrm{deg}f_{n-1})$.
\end{itemize}
Given an exact category $\mathcal{A}$, we denote the dg category of bounded complexes over $\mathcal{A}$ by $\mathscr{C}^b(\mathcal{A})$ and its subcategory of acyclic complexes by $\mathscr{AC}^b(\mathcal{A})$. The mixed complex of the exact category $\mathcal{A}$ is defined as the mapping cone of the embedding of complexes $\mathscr{AC}^b(\mathcal{A})\hookrightarrow\mathscr{C}^b(\mathcal{A})$. The mixed complexes can be also be viewed as dg modules over the dg algebra $k[\epsilon]/(\epsilon^2)$ with $\mathrm{deg}\epsilon=1$ and trivial differentials. In the derived category of $\Lambda$, we then define the cyclic homologies $HC$, $HN$, $HP$ of $\mathcal{A}$ by applying functors $-\otimes_\Lambda^{\mathbf{L}}k$, $\mathbf{R}\Hom_\Lambda(k,-)$ and $\mathbf{R}\varprojlim P_k[-2n]\otimes_\Lambda-$ to its mixed complex and taking the homology. So in the sense of Keller, all the cyclic homologies are ``unified" by the mixed complex \cite{Ke2}. In the following, when used without further specification, the word ``cyclic homology" of an exact category will really stand for the mixed complex of it.

Recall that for a smooth projective variety $X$, the mixed complex can be defined on the category of coherent sheaves over $X$ \cite{Ke3}.
\begin{prop}The cyclic homology of $\vc$ is the same as the cyclic homology of $X$.\end{prop}
\begin{proof} By the equivalence of categories in the previous section, we only consider the heart $\mathcal{C}^\theta$ of some t-structure of $D^b(X)$. The proof is similar to the proof of the statement $D^b(\mathcal{C}^\theta)=D^b(X)$.

It is known that, for elliptic curves, every bounded t-structure is given by a cotilting torsion pair.
In our case, it is $(\mathrm{Coh}_{>\theta},\,\mathrm{Coh}_{\leq\theta})$ \cite{GKR}. So $(\mathrm{Coh}_{\leq\theta}[1],\,\mathrm{Coh}_{>\theta})$ is a tilting torsion pair in $\mathcal{C}^\theta$.

$\mathrm{Coh}_{\leq\theta}$ is an exact category. Moreover, by the lemma and proposition and their duals of \cite{BvdB}, we have $D^b(X)=D^b(\mathrm{Coh}_{\leq\theta})$ and $D^b(\mathcal{C}^\theta)=D^b(\mathrm{Coh}_{\leq\theta}[1])$.

By the main theorem of \cite{Ke2}, we have the quasi-isomorphisms of cyclic complexes induced by the inclusions: $C(\mathrm{Coh}_{\leq\theta})\tilde{\rightarrow}C(X)$ and $C(\mathrm{Coh}_{\leq\theta}[1])\tilde{\rightarrow}C(\mathcal{C}^\theta)$. So there is a quasi-isomorphism $C(\mathcal{C}^\theta)\simeq C(X)$.
\end{proof}

\begin{prop}The functor $F$ in proposition 3.1 induces an injection $$C(\vc) \rightarrowtail C(\pr)\tilde{\rightarrow}C(\mathcal{A}_\theta).$$\end{prop}
\begin{proof} The last quasi-isomorphism is due to the theorem of McCarthy on the special homotopy \cite{Mc}.

To show $F$ induces the injection $C(\vc)\rightarrowtail C(\pr)$, we need to introduce a third category $\pr^\nabla$. $\pr^\nabla$ has the same objects as $\vc$. For any $(E_1,\nabla_1)$, $(E_2,\nabla_2)$ in $\pr^\nabla$, we define $\Hom_{\pr^\nabla}((E_1,\nabla_1),(E_2,\nabla_2))=\Hom_{\mathcal{A}_\theta}(E_1,E_2)$.

We can define a natural functor $\pr^\nabla\rightarrow\pr$ simply as the forgetful functor
that drops the $\nabla$'s.
It is obvious that this functor is fully faithful. Moreover, it is surjective on objets (hence essentially surjective). This is an equivalence of categories. This gives the quasi-isomorphism $C(\pr^\nabla)\tilde{\rightarrow}C(\pr)$ by Keller's theorem again.

On the other hand, $\vc$ is a subcategory of $\pr^{\nabla}$ (but not full). So the functor $F:\vc\rightarrow\pr$ is just the composition of $\vc\hookrightarrow\pr^\nabla\rightarrow\pr$. According to the definition of mixed complex of exact categories, we see $\mathscr{C}^b(\vc)\subset\mathscr{C}^b(\pr^\nabla)$ and $\mathscr{AC}^b(\vc)\subset\mathscr{AC}^b(\pr^\nabla)$. It follows that $C(\vc)\hookrightarrow C(\pr^\nabla)$. Thus $F$ induces the morphism of complexes: $$C(\vc)\hookrightarrow C(\pr^\nabla)\tilde{\rightarrow}C(\pr).$$
\end{proof}

\section{Motives of noncommutative tori}
\subsection{The(weak)t-structures}In the previous section, we recalled the known fact that the category of holomorphic bundles on the noncommutative torus is equivalent to a heart of a nonstandard t-structure on the derived category of some elliptic curve. We know that an elliptic curve, as a smooth projective variety, can be used to define an objet in the category of pure motives and so an object in the category of noncommutative pure motives. In this section we propose a way to extend the concept of t-structure to the frame of noncommutative motives. In this way, we can obtain objects with the properties of
noncommutative tori in the category of noncommutative motives.

Assume $\mathcal{C}$ is a dg category which has a shift functor $E[i]$ and such that for any closed and degree $0$ morphism $f:E\rightarrow F$ in $\mathcal{C}$ there is a mapping cone $Cone(f)$ in $\mathcal{C}$.

\begin{defn}
Let $\mathcal{C}^{\leq0}$ and $\mathcal{C}^{\geq0}$ be full subcategories of $\mathcal{C}$, $\mathcal{C}^{\leq n}=\mathcal{C}^{\leq0}[-n]$ and $\mathcal{C}^{\geq n}=\mathcal{C}^{\geq0}[-n]$. One says that $(\mathcal{C}^{\leq0},\,\mathcal{C}^{\geq0})$ is a \textbf{(weak)-t-structure} if:
\begin{itemize}
\item[(i)] $\mathcal{C}^{\leq-1}\subset\mathcal{C}^{\leq0}$ and $\mathcal{C}^{\geq1}\subset\mathcal{C}^{\geq0}$.
\item[(ii)] For $X\in \mathrm{Ob}(\mathcal{C}^{\leq0})$ and $Y\in\mathrm{Ob}( \mathcal{C}^{\geq1})$, $\mathrm{Hom}(X,\,Y)=0$ in $H^0(\mathcal{C})$.
\item[(iii)] For any $X\in\mathrm{Ob}(\mathcal{C})$, there is a triangle $X_0\rightarrow X\rightarrow X_1\rightarrow X_0[1]$ in $\mathcal{C}$ with $X_0\in\mathrm{Ob}(\mathcal{C}^{\leq0})$ and $X_1\in\mathrm{Ob}( \mathcal{C}^{\geq1})$, which is quasi-equivalent to a triangle having the form $E\rightarrow F\rightarrow Cone(f)\rightarrow E[1]$ for some closed degree $0$ morphism $f:E\rightarrow F$.
\end{itemize}
\end{defn}
\begin{rem} Here, we may use ``weak" in the sense of quasi-equivalence. As usual, the word 	``strong" stands for the dg-equivalence.\end{rem}

Recall that a triangulated category $\mathcal{K}$ is dg enhanced, if there is a pretriangulated category $\mathcal{C}$ with an equivalence $H^0(\mathcal{C})\rightarrow \mathcal{K}$ of triangulated categories. We know that the category of $H^0(\mathcal{C})$ has the same objects as $\mathcal{C}$. So, if $(\mathcal{K}^{\leq0},\,\mathcal{K}^{\geq0})$ is a t-structure of triangulated categories on $\mathcal{K}=H^0(\mathcal{C})$, then we get two full subcategories of $\mathcal{C}$ generated by the objects of $\mathcal{K}^{\leq0}$ and $\mathcal{K}^{\geq0}$, and we use the notation $(\mathcal{C}^{\leq0},\,\mathcal{C}^{\geq0})$ to denote them. We will see that $(\mathcal{C}^{\leq0},\,\mathcal{C}^{\geq0})$ is a t-structure on the dg category $\mathcal{C}$. Actually, condition (i) and (ii) in our definition are naturally satisfied. For condition (iii), because $\mathcal{K}$ is pretriangulated, every distinguished triangle is induced in that way. So (iii) holds. Hence, if $\mathcal{K}$ is dg enhanced, the t-structure on $\mathcal{K}$ always induces a t-structure on its dg enhancement.

The following proposition is a standard result for t-structures.
\begin{prop}For any $X\in\mathcal{C}^{\leq n}$ (resp. $X\in\mathcal{C}^{\geq n})$, $Y\in\mathcal{C}$ and $k\in\mathbb{Z}$, there is a $Y_0\in\mathcal{C}^{\leq n+k}$ (resp. $Y_1\in\mathcal{C}^{\geq n+k}$), such that $$\h^k(X,\,Y_0)\tilde{\longrightarrow}\h^k(X,Y),$$
$$\big(\text{resp.}\quad\h^k(Y_1,\,X)\tilde{\longrightarrow}\h^k(Y,X).\big)$$\end{prop}

\begin{proof}First, we notice that if $(\mathcal{K}^{\leq0},\mathcal{K}^{\geq0})$ is a t-structure of a triangulated category $\mathcal{K}$, then for any $X\in\mathcal{K}$, there is an exact triangle $$X_0\rightarrow X[n]\rightarrow X_1\rightarrow X_0[1]$$ with $X_0\in\mathcal{K}^{\leq0}$ and $X_1\in\mathcal{K}^{\geq1}$. Then, we have $$X_0[-n]\rightarrow X\rightarrow X_1[-n]\rightarrow X_0[-n+1]$$ with $X_0[-n]\in\mathcal{K}^{\leq0}[-n]=\mathcal{K}^{\leq n}$ and $X_1[-n]\in\mathcal{K}^{\geq1}[-n]=\mathcal{K}^{\geq n+1}$. Therefore, $(\mathcal{K}^{\leq n},\mathcal{K}^{\geq n})$ is also a t-structure for $\mathcal{K}$. Hence the standard argument gives the adjoint functors $\tau^{\leq n},\,\tau^{\geq n}$:
$$\Hom_{\mathcal{K}^{\leq n}}(X,\,\tau^{\leq n}Y)\tilde{\longrightarrow}\Hom_{\mathcal{K}}(X,Y),\,\text{with}\,X\in\text{Ob}(\mathcal{K}^{\leq n}),\,Y\in\text{Ob}(\mathcal{K})$$ and
$$\Hom_{\mathcal{K}^{\geq n}}(\tau^{\geq n}X,\,Y)\tilde{\longrightarrow}\Hom_{\mathcal{K}}(X,Y)\,\text{with}\,X\in\text{Ob}(\mathcal{K}),\,Y\in\text{Ob}(\mathcal{K}^{\geq n}).$$

Now let $X$ be an object of $\mathcal{C}^{\leq n}$. Then $\h^k(X,Y)=\h^0(\Hom(X[-k],Y)),$ with $X[-k]$ is in $\mathcal{C}^{\leq n+k}$. Thus, if we set $Y_0=\tau^{\leq n+k}Y$, we have $$\Hom(X[-k],\,Y_0)\tilde{\longrightarrow}\Hom(X[-k],Y).$$ Moreover, we have $$\h^k(X,Y_0)\tilde{\longrightarrow}\h^k(X,Y).$$

The other statement is proven similarly.
\end{proof}

\subsection{Decomposition of motivic Galois groups}
Let $\mathcal{D}$ be a triangulated category which is dg enhanced and idempotent complete. Thus, there is a dg pre-triangulated category $\mathcal{A}$ such that $H^0(\mathcal{A})=\mathcal{D}$. As To\"{e}n and Vaqui\'{e} showed in \cite{TV}, there is an equivalence $\mathcal{D}\rightarrow \mathcal{D}_c(\mathcal{A})$.

\begin{thm}If $(\mathcal{D}^{\leq0},\,\mathcal{D}^{\geq0})$ is a t-structure on $\mathcal{D}$, then it also gives a (weak-)t-structure on $\mathcal{A}$, which may be denoted by $(\mathcal{D}^{\leq0}_{dg},\mathcal{D}^{\geq0}_{dg})$. Then for any additive invariant $F:\mathcal{A}\rightarrow\mathcal{B}$, we have $F(\mathcal{D}^{\leq0}_{dg})\oplus F(\mathcal{D}^{\geq1}_{dg})=F(\mathcal{A})$.\end{thm}

\begin{proof}It is sufficient to show that, for any $\mathcal{U}\in \mathrm{dgcat}_k$,$$\mathrm{K_0rep}(\mathcal{U},\mathcal{D}^{\leq0}_{dg})\oplus \mathrm{K_0rep}(\mathcal{U},\mathcal{D}^{\geq1}_{dg})=\mathrm{K_0rep}(\mathcal{U},\mathcal{A}).$$
Recall $\mathrm{rep}(\mathcal{U},\mathcal{A})\subset\mathcal{D}(\mathcal{U}^{op}\otimes\mathcal{A})$ is the sub-triangulated category of bimodules $X$ such that for any $x\in\mathcal{U}$, $X(\,\,,x)\in\mathcal{D}_c(\mathcal{A})$.

This theorem is proven if we notice the fact that $H^0(\mathcal{A})\rightarrow\mathcal{D}_c(\mathcal{A})$ is an equivalence.  For any $X\in\mathrm{Ob}(H^0(\mathcal{A}))$, there is an exact triangle $$X_0\rightarrow X\rightarrow X_1\rightarrow X_0[1]$$ in $H^0(\mathcal{A})$ with $X_0$ in $\mathcal{D}^{\leq0}=H^0(\mathcal{D}^{\leq0}_{dg})\hookrightarrow H^0(\mathcal{A})$ and $X_1$ in $\mathcal{D}^{\geq1}=H^0(\mathcal{D}^{\geq1}_{dg})\hookrightarrow H^0(\mathcal{A})$. By the equivalence, this gives an exact triangle in $\mathcal{D}_c(\mathcal{A})$, hence in $\mathrm{rep}(\mathcal{U},\mathcal{A})$. After passing to the Grothendieck group, the exact triangle becomes the desired direct sum decomposition.\end{proof}

It is known \cite{Ke1}, \cite{CT}, the periodic cyclic homology $HP: \mathrm{NChow}\rightarrow \mathrm{sVect}(k)$ is an additive invariant. Also, Marcolli and Tabuada showed in \cite{MT2}, the category $\mathrm{NChow}$ is super-Tannakinan with the super-fibre functor $HP$.

Let $X$ be the elliptic curve above and $D^b_{dg}(X)$ the object associated to $X$ in NChow. Actually, by the Hodge conjecture, the Tannakian subcategory $M_X$ generated by $D^b_{dg}(X)$ in NChow is Tannakian with the fibre functor $HP$ \cite{MT2} and the motivic Galois group $Gal(D^b_{dg}(X))$. The t-structure $(\mathbf{D}^{\leq0}_{dg},\,\mathbf{D}^{\geq0}_{dg})$ on $D^b_{dg}(X)$ gives the direct sum $$HP(\mathbf{D}^{\leq0}_{dg})\oplus HP(\mathbf{D}^{\geq1}_{dg})=HP(D^b_{dg}(X)).$$ We define $HP^{\leq0}$ as the composition $D^b_{dg}(X)\rightarrow HP(D^b_{dg}(X))\rightarrow HP(\mathbf{D}^{\leq0}_{dg})$ and $HP^{\geq0}$ as the composition $D^b_{dg}(X)\rightarrow HP(D^b_{dg}(X))\rightarrow HP(\mathbf{D}^{\geq0}_{dg})$. Moreover, $HP^{\leq0}$ and $HP^{\geq0}$ induce two subgroups $G^{\leq0}$, $G^{\geq0}$ of $Gal(D^b_{dg}(X))$.

\begin{cor}There are two subgroups $G^{\leq0}$, $G^{\geq0}$ of $Gal(D^b_{dg}(X))$ associated to the t-structure on $D^b_{dg}(X)$ and $G^{\leq0}\oplus G^{\geq0} \subset Gal(D^b_{dg}(X))$\end{cor}


\begin{thebibliography}{100}
\bibitem[BvdB]{BvdB} A. Bondal \& M. van den Bergh, \emph{Generators and representability of functors in commutative and noncommutative geometry.} Mosc. Math. J.  3  (2003),  no. 1, 136-258.
\bibitem[BK]{BK} A. Bondal \& M. Kapranov, \emph{Enhanced triangulated categories.} Math. USSR Sbornik 70 (1991), no. 1, 93-107.
\bibitem[CT]{CT} D.-C. Cisinski \& G. Tabuada, \emph{Symmetric monoidal structure on non-commutative motives.} J. K-Theory  9  (2012),  no. 2, 201-–268.
\bibitem[D]{D} P. Deligne, \emph{Cat\'{e}gories tensorielles.} Moscow Math. J. 2 (2002), no. 2, 227-248.
\bibitem[GKR]{GKR} A. Gorodentsev, S. Kuleshov and A. N. Rudakov, \emph{t-stabilities and t-structures on triangulated categories.} Izv. Ross. Akad. Nauk Ser. Mat. 68 (2004), 117-150.
\bibitem[H]{H} M. Hanamura, \emph{Mixed motives and algebraic cycles, I.} Math. Res. Lett. 2 (1995), no. 6, 811-821.
\bibitem[J]{J} U. Jannsen, \emph{Motives, numerical equivalence, and semi-simplicity.} Invent. Math.  107  (1992),  no. 3, 447-452.
\bibitem[Ke1]{Ke1} B. Keller, \emph{On differential graded categories.}  International Congress of Mathematicians. Vol. II,  151-190, Eur. Math. Soc., Zürich, (2006).
\bibitem[Ke2]{Ke2} B. Keller, \emph{On the cyclic homology of exact categories.} J. Pure Appl. Algebra  136 (1999),  no. 1, 1-56.
\bibitem[Ke3]{Ke3} B. Keller, \emph{On the cyclic homology of ringed spaces and schemes.} Doc. Math.  3  (1998), 231-259.
\bibitem[Ko]{Ko1} M. Kontsevich, \emph{Noncommutative motives}. Talk at IAS. Oct., 2005.
\bibitem[L]{L} M. Levine, \emph{Mixed Motives,Math.} Surveys and Monographs 57 AMS, Prov. 1998.
\bibitem[Mc]{Mc} R. McCarthy, \emph{The cyclic homology of an exact category.} J. Pure Appl. Algebra  93  (1994),  no. 3, 251-296.
\bibitem[Mi]{Mi} J. S. Milne, \emph{Motives— Grothendieck's Dream}. Notes available online: http://www.jmilne.org/.
\bibitem[MT1]{MT1} M. Marcolli \& G. Tabuada, \emph{Noncommutative motives, numerical equivalence, and semisimplicity.} American Journal of Mathematics, 136 (2014), no. 1, 59-75
\bibitem[MT2]{MT2} M. Marcolli \& G. Tabuada, \emph{Noncommutative numerical motives, Tanakian structures, and motivic Galois groups.} Available at arXiv:1110.2438.
\bibitem[MT3]{MT3} M. Marcolli \& G. Tabuada, \emph{Noncommutative Artin motives.} Selecta Math. (N.S.)  20  (2014),  no. 1, 315-358.
\bibitem[MVW]{MVW} C. Mazza, V. Voevodsky \& C. Weibel, \emph{Lecture notes on motivic cohomology.} Clay Monographs in Math. 2. Amer. Math. Soc. 2006.
\bibitem[PS]{PS}  A. Polishchuk \& A. Schwarz, \emph{Categories of holomorphic vector bundles on noncommutative two-tori.} Comm. Math. Phys.  236  (2003),  no. 1, 135-159.
\bibitem[P]{P} A. Polishchuk, \emph{Classification of holomorphic vector bundles on noncommutative two-tori.} Doc. Math.  9  (2004), 163-181.
\bibitem[R]{R} M. A. Rieffel, \emph{The cancellation theorem for projective modules over irrational rotation C*-algebras.} Proc. London Math. Soc. (3)  47  (1983),  no. 2, 285-302.
\bibitem[T1]{Tab1} G. Tabuada, \emph{Higher K-theory via universal invariants.} Duke Math. J. 145 (2008), no. 1, 121-206.
\bibitem[T2]{Tab2} G. Tabuada, \emph{Chow motives versus noncommutative motives.} J. Noncommut. Geom.  7  (2013),  no. 3, 767-786.
\bibitem[TV]{TV} B. To\"{e}n \& M. Vaqui\'{e} \emph{Moduli of objects in dg-categories.} Ann. Sci. École Norm. Sup. (4)  40  (2007),  no. 3, 387-444.
\bibitem[V]{V} V. Voevodsky, \emph{Triangulated categories of motives over a field.} Cycles, transfer and motivic homology theories, pp. 188-238, Annals of Mathematical Studies, Vol. 143, Princeton, 2000.

\end{thebibliography}

\end{document}